\documentclass[11pt]{article}

\usepackage{amssymb}
\usepackage{amsmath}
\usepackage{subfigure}
\usepackage{graphicx}
\usepackage{algorithm}
\usepackage{algpseudocode}
\usepackage{amsthm}
\usepackage{enumitem}
\usepackage{bm}
\usepackage{here}
\usepackage{hyperref}

\newtheorem{dfn}{Definition}

\newtheorem{cla}[dfn]{Claim}

\newcommand{\argmin}{\mathop{\rm arg~min}\limits}
\newcommand{\argmax}{\mathop{\rm arg~max}\limits}
\newcommand{\trans}{\mathsf{T}}
 \newcommand{\xR}{\mathbb{R}}
 \newcommand{\xN}{\mathbb{N}}
 \algrenewcommand{\Return}{\State\algorithmicreturn~}
 \newcommand{\busub}[1]{\noindent \underline{\textbf{#1}}\\}

\title{Sample-Cluster-Select: A new framework to obtain diverse approximate solutions of combinatorial optimization problems\thanks{This work was supported by JSPS KAKENHI Grant Number 24K17472, MEXT Quantum Leap Flagship Program (MEXT Q-LEAP) Grant Number JPMXS0118069605, and MEXT KAKENHI Grant Number 20H05962.}} %% Article title

\author{Susumu Hashimoto, Takeaki Uno\\
    National Institute of Informatics\\
    \url{hashimoto_s@nii.ac.jp}, \url{uno@nii.ac.jp}}
\date{\today}
\begin{document}
\maketitle
\begin{abstract}
  When solving real-world problems, practitioners often hesitate to implement solutions obtained from mathematical models, especially for important decisions.
  This hesitation stems from practitioners' lack of trust in optimization models and computational results.
  To address these challenges, we propose Sample-Cluster-Select (SCS) for solving practical combinatorial optimization problems under the assumption of potentially acceptable solution set.
  SCS first samples the potential solutions, performs clustering on these solutions, and selects a representative solution for each cluster.
  SCS aims to build trust by helping users understand the solution space through multiple representative solutions, while simultaneously identifying promising alternatives around these solutions.
  We conducted experiments on randomly generated instances, comparing SCS against multi-start local search and $k$-best algorithms where efficient methods exist, or evolutionary algorithms otherwise.
  The results show that SCS outperforms multi-start local search and $k$-best algorithms in most cases, while showing competitive performance against evolutionary algorithms, though not surpassing some of their variants.
  Most importantly, we found that the clustering approach provides insights into solutions that are difficult to obtain with existing methods, such as local structures of similar potential solutions and neighboring solutions of representative solutions.
  Thus, our approach helps practitioners understand the solution space better, thereby increasing their confidence in implementing the proposed solutions.
\end{abstract}

    \section{Introduction}
    \label{sec:intro}
    Combinatorial optimization models can express various real-world decision-making situations.
    Since many classical combinatorial optimization problems and their variants are $NP$-hard, high-performance heuristics, such as Evolutionary Algorithm (EA), have been developed to obtain optimal or suboptimal solutions.
    However, users or practitioners often hesitate to directly implement these solutions, even though multiple diverse solutions are prepared.
    This hesitation stems from practitioners' lack of trust in optimization models and computational results.
    To overcome this challenge, we propose an approach that not only finds multiple high-quality solutions but also reveals the characteristics of neighboring solutions around them, thereby demonstrating to practitioners the range of possible solutions.
    Before presenting our approach in detail, let us discuss the root causes of this lack of trust.
    Specifically, we give three critical issues:
    \begin{enumerate}[label=Issue \arabic*,ref=\arabic*]
        \item The optimal solutions of the model may not be acceptable to the user. \label{iss:not_opt}
        \item Even when an acceptable solution is obtained, it may be unclear why it is favorable. \label{iss:blackbox}
        \item The user might suspect that there are better solutions other than the candidates presented. \label{iss:suspect}
    \end{enumerate}

    First, we elaborate on Issue \ref{iss:not_opt}. Consider the problem of finding the route for cycling from the workplace to home. Suppose there are two options: a 1 km shortest path and a 1.5 km detour. Without additional information, the user would likely select the shortest path. However, the user might prefer the longer route if the shortest path involves steep hills or if the detour passes by a supermarket.
    Unless the objective function almost expresses the user's preferences, solutions that are not (sub)optimal may still be acceptable for the user.
    Conversely, not all feasible solutions are necessarily appealing.
    For instance, a route that significantly deviates from and returns near the starting point, resembling the Greek letter $\Omega$, is unlikely to be preferable for the user. If such a route is often considered satisfying, then modeling the problem as a shortest path problem may be inappropriate.

    Regarding Issue \ref{iss:blackbox}, while metaheuristics like EA can quickly obtain (sub)optimal solutions, their mechanism is opaque and lacks interpretability, making it difficult for users to trust the solutions.
    This black-box nature of the algorithm obscures the reasoning behind why certain solutions are considered favorable.
    Furthermore, even when we get a satisfying solution, we miss opportunities to gain valuable insights such as which variables should be set to 1, the relationships between variable assignments (e.g., $x_A=1$ implies that setting $x_B$ to 1 is better), or information about the alternative when the prepared solutions become infeasible.
    Such knowledge not only enhances the credibility of the solutions by providing a clear understanding of their underlying structure, but also helps us find acceptable solutions when circumstances have changed and the original solutions are no longer applicable.

    Issue \ref{iss:suspect} also leads to the problem that users are reluctant to trust the provided solutions.
    For example, consider buying a house through an online real estate service.
    While such services typically aim to maximize the probability of purchase by recommending properties ranked according to various preferences with user-specified constraints, users rarely decide based solely on these top-ranked recommendations.
    Instead, they may review lower-ranked options and try different search conditions to explore other available properties before deciding.
    This behavior implies that presenting only multiple solutions does not necessarily lead to user satisfaction or immediate selection from among the options, especially for important decision-making.

    In this paper, we study an optimization problem that takes into account Issue \ref{iss:not_opt} and present an algorithmic framework to overcome Issues \ref{iss:blackbox}--\ref{iss:suspect}.
    The essence of our proposal is sampling numerous potential solutions and aggregating them to reveal practical characteristics of good solutions.
    Suppose that the original optimization problem is as follows:
    \begin{align}
        \text{minimize} \quad &f(x) \nonumber\\
        \text{subject to} \quad &x \in \mathcal{F}.
        \label{eq:original}
    \end{align}
    To address our issues, we assume a set of solutions $\mathcal{G} \subseteq \mathcal{F}$ that are potentially to be acceptable to users.
    Then, we consider the following problem:
    \begin{align}
        \text{maximize} \quad &D(X) \nonumber\\
        \text{subject to} \quad &X \subseteq \mathcal{G},
        \label{eq:main}
    \end{align}
    where $D: 2^{\mathcal{G}} \to \xR_{+}$ represents a function evaluating the diversity of a solution set.
    Intuitively, Problem \eqref{eq:main} aims to find the most diverse potentially acceptable solution set.

    Then, this paper proposes a framework called Sample-Cluster-Select (SCS) algorithm to solve \eqref{eq:main}.
    SCS consists of three main steps: (i) random sampling of potentially acceptable solutions, (ii) clustering the sampled solutions by their similarity, and (iii) selecting representative solutions, one from each cluster.
    Aggregating good solutions within a cluster can reveal potential structures around the representative solution, addressing Issues \ref{iss:blackbox} and \ref{iss:suspect}.

    This paper uses local optimal solutions as the potential solutions in the experiment and employ a local search algorithm as the sampling method in (i).
    Local optimal solutions are sometimes used in practice because they can easily obtain reasonably good solutions in typical neighborhoods, and besides, the concept is intuitive to users.
    Thus, they are suitable for our purpose.
    In (ii), we employ a micro-clustering algorithm by data polishing \cite{Uno2017}.
    This algorithm is appropriate for SCS because it clusters only similar solutions and allows outliers or unlabeled solutions.

    The performance and the above feature of SCS are examined by numerical experiments for random instances of shortest path problem (SP), traveling salesman problem (TSP), and set packing problem (SPP).
    Since the most appropriate diversity function cannot be defined, we evaluate the algorithm's performance using several measures, such as solution dispersions and coverage indicators for other solutions.
    Additionally, we demonstrate that superimposing solutions within clusters can yield insights that are unattainable through existing methods.
    We utilize positive-weighted SP with grid graph, SP with Euclidean unit disk graph, Euclidean TSP, and rectangle SPP as benchmark problems. These problems allow an intuitive visualization of the aggregated solutions for evaluating our framework against Issues \ref{iss:blackbox}--\ref{iss:suspect}. We employ rectangle SPP over conventional SPP because it represents solutions as rectangles on a grid graph, making them easy to interpret.

    Next, we introduce the structure of this paper.
    Section \ref{sec:prev} summarizes the related previous research.
    Section \ref{sec:probls} denotes the problem descriptions and local search algorithms used in the experiments as SCS's sampling methods.
    The experimental results are shown in Section \ref{sec:exp}, and Section \ref{sec:con} gives a conclusion.

    Finally, we give notations used in this paper.
    This paper denotes $\xN$ by the set of natural numbers, $\xR$ by the set of real numbers, $\xR_{+}$ by the set of non-negative real numbers, and $\xR_{++}$ by the set of positive real numbers. $\lceil r \rceil$ represents the smallest integer greater than or equal to $r\in \xR$. $(e_{ij})_{ij}$ denotes the matrix whose $(i, j)$-th element is $e_{ij}$.

    \section{Previous research}
    \label{sec:prev}
    This section describes the previous research related to this paper.
    First, this section reviews the existing problem settings that require multiple solutions to combinatorial optimization problems for the examined problems.
    The $k$-best solution problems aim to enumerate $k$ best optimal solutions for arbitrary integer $k$.
    Yen \cite{Yen1970} establishes a $O(k|V|(|E| + |V| \log |V|))$-time algorithm for $k$-best SP with non-negative edge weights.
    Hershberger et al. \cite{Hershberger2007} also propose a $O(k|V|(|E| + |V|\log |V|))$-time algorithm for SP with non-negative edge weights, which runs in $O(k(|E| + |V|\log |V|))$-time in optimistic cases.
    Similar to the $k$-best solution problems, multi-solution problems are also studied.
    This problem requires all the optimal solutions.
    As a related study, Huang et al. \cite{Huang2019} propose a Niching Memetic Algorithm for multi-solution TSP.
    Similar to the multi-solution problem, Ulrich and Thiele \cite{Ulrich2011} investigate the diversity maximization problem with an objective barrier, which aims to find a solution set with maximum diversity while ensuring that all solutions' objective values are better than a given barrier, and proposes NOAH framework of EA algorithms.
    Nikfarjam et al. \cite{Nikfarjam2021} apply EAX crossover, which is known as a strong EA operator for TSP, to NOAH and develop an algorithm that efficiently solves the diversity maximization problem of Euclidean TSP.
    Multi-objective optimization is well studied as another modeling approach to obtain diverse solutions.
    This problem has multiple objective functions, and the goal is to obtain representative solutions of the Pareto optimal solutions.
    Multi-objective problems of SP \cite{Maristany2021}, TSP \cite{George2020}, and SPP \cite{Delorme2010} are investigated.
    Sharma and Kumar \cite{Sharma2022} give a detailed survey of this type of problem.

    Then, the paper briefly reviews the clustering methods applicable to SCS and discusses their validity.
    $k$-means \cite{MacQueen1967} is one of the most widely used clustering methods, and it labels all data under a given number of clusters.
    However, in the context of this paper, it is crucial to aggregate similar data or solutions; therefore, labeling all the data is undesirable.
    A well-known method for this purpose is DBSCAN \cite{Ester1996}.
    This method extracts dense subsets of data in terms of similarity to clusters while allowing for outliers. DBSCAN is an exclusive clustering method where each data point belongs to exactly one cluster. A limitation of this method is that the results depend on the data's indexing.
    Micro-clustering with data polishing \cite{Uno2017} has the same motivation as DBSCAN and solves the indexing problem. It detects approximate pseudo-cliques in the adjacency graph of the data, allowing for outliers and overlapping clusters.
    Cluster overlapping is advantageous for SCS because we aim to find common structures in dense potential solutions, and determining their label itself is not essential.
    Therefore, micro-clustering with data polishing is a suitable clustering method for SCS.

    Finally, we describe the differences between the proposal of this paper and the aforementioned existing approaches.
    First, $k$-best and multi-solution approaches require optimal solutions for the target problem to be obtained in a reasonable time. This requirement significantly limits their applicability, especially for $NP$-hard problems, to only small-scale instances.
    In contrast, SCS can be applied to any problem where potential solutions can be efficiently sampled.
    As for the diversity maximization problem, this has already been discussed in Section \ref{sec:intro}.
    While multi-objective optimization can handle multiple types of preferences, these preferences must be mathematically defined as objective functions.
    This requirement makes it unsuitable for handling tacit knowledge or personal preferences that cannot be explicitly formulated.
    In contrast, \eqref{eq:main} maximizes the diversity of potential solutions, expecting to find solutions that fortunately align with implicit preferences.
    Moreover, SCS provides clusters around each representative solution, enabling users to explore alternative solutions within promising clusters even when their representative solutions are unacceptable.

    \section{Proposed framework}
    \label{sec:scs}
    This section describes the proposed framework to solve Problem \eqref{eq:main}, Sample-Cluster-Select (SCS) algorithm.
    First, we give an overview of SCS.
    SCS literally consists of three steps: sampling, clustering, and selecting.
    The sampling step picks potential solutions.
    Then, the clustering step groups the solutions using the given dissimilarity function of solutions as a distance function.
    Finally, the selecting step chooses a solution from each cluster and outputs all of them.
    SCS has two advantages.
    First, it does not necessarily require complicated parameter tuning or algorithm design.
    Second, it can obtain new insights by sampling many solutions and clustering them.
    These points are shown in the numerical experiments.

    Next, this paper provides a detailed description of the framework configuration used in the experiments.

    \noindent \underline{\textbf{Input}}
    The inputs of SCS are as follows:
    \begin{itemize}
        \item $n\in \xN$, $0<\delta<1$, $\beta \in \xR_{++}$
        \item A potential solution set $\mathcal{G}$
        \item A dissimilarity function of the solutions $d : \mathcal{G}\times \mathcal{G}\to [0, 1]$
        \item A random sampling method from $\mathcal{G}$
        \item Parameters for micro-clustering
        \item A selection criterion used in the selection step
    \end{itemize}

    \noindent \underline{\textbf{Sampling step}}
    The sampling step samples $n$ potential solutions from $\mathcal{G}$.
    This paper utilizes a simple multi-start local search (based) algorithm or independent repetition of a randomized local search algorithm for the original problem \eqref{eq:original}.
    To sample various solutions, we employ the first improvement strategy, which updates the current solution whenever a better neighboring solution is found.

    \noindent \underline{\textbf{Clustering step}}
    The clustering step first calculates the dissimilarity or distance between each pair of the sampled solutions.
    Next, it constructs a graph with the obtained solutions as vertices and edges connecting pairs of similar solutions.
    Let $S = (S_1, \dots, S_{n})\in  \mathcal{G}^{n}$ be the set of solutions obtained by the sampling step.
    The adjacency graph $G = (V, E)$ of the solutions is calculated as follows:
    \begin{enumerate}[label=Step \arabic*., ref=\arabic*]
        \item $V = \{1, 2, \dots, n\}$.
        \item Enumerate all non-zero dissimilarities between pairs of solutions, and let $n_{+}$ be the number of such dissimilarities.
        \item Let $\tau$ be the distance of the $\lceil \delta n_{+} \rceil$-th smallest dissimilarity.
        \item $E = \{(i, j) \in V \times V : d(S_i, S_j) \le \tau\}$.
    \end{enumerate}

    Next, we cluster the solutions using the adjacency graph and the micro-clustering algorithm by data polishing \cite{Uno2017}.
    This algorithm enumerates all approximate pseudo-cliques in the graph as clusters. See Uno et al. \cite{Uno2017} for the details.
    Note that this algorithm allows the same node to be included in multiple clusters.

    \noindent \underline{\textbf{Selecting step}}
    Finally, we describe the selecting step.
    It aims to choose a solution from each cluster and outputs all of them.
    This step begins with two setup operations: (1) eliminating identical sampled solutions and (2) determining a single cluster assignment for solutions assigned to multiple clusters.
    For solutions with multiple assignments, we select the final cluster in a way that balances cluster sizes.
    Let $S = (S_1, S_2, \dots, S_{n}) \subseteq \mathcal{G}$ be the sampled solutions and $C = (C_1, C_2, \dots, C_{n_c}) \in (2^{S})^{n_c}$ be the clusters, where $n_c$ is the number of clusters.
    The setup method is as follows:
    \begin{enumerate}[label=Step \arabic*., ref=\arabic*]
        \item Remove all identical solutions except one representative
        \item Pick a solution that belongs to multiple clusters
        \item Pick one of the smallest clusters containing this solution
        \item Remove this solution from all other clusters
        \item If solutions with multiple assignments still exist, return to Step 2; otherwise, terminate
    \end{enumerate}

    Then, the selection method is described in Algorithm \ref{alg:selection}.
    \begin{algorithm}[H]
        \caption{Selection algorithm}\label{alg:selection}
        \begin{algorithmic}[1]
            \Require $S = (S_1, \dots, S_{n}) \in  \mathcal{G}^{n}$, $C = (C_1, C_2, \dots, C_{n_c}) \in (2^{S})^{n_c}$, $d : \mathcal{G} \times \mathcal{G} \to \xR_{+}$, a selection criterion
            \Ensure $S_{SEL} \subseteq S$
            \State {Compute $C^\prime = (C^\prime_1, C^\prime_2, \dots, C^\prime_{n_c})\in (2^{S})^{n_c}$ from $C$ by the setup method}
            \While{$\max_{1\le i\le n_c} |C^\prime_i| > 1$} \label{line:sel-start}
            \State $j \gets  \argmax_{1\le i\le n_c}|C^\prime_i|$
            \State {Pick a solution $\bar{S}$ from $C^\prime_i$ by the selection criterion}
            \State $C^\prime_i \gets C^\prime_i\setminus \{\bar{S}\}$
            \EndWhile \label{line:sel-end}
            \Return $\bigcup_{1\le i\le n_c} C^\prime_i$
        \end{algorithmic}
    \end{algorithm}

    Algorithm \ref{alg:selection} selects one of the biggest clusters and then greedy eliminates a solution from the cluster to maximize a certain diversity measure. It continues this process until every cluster has a single solution at most.
    Let $\tilde{C}$ be the selected biggest cluster and $\tilde{S}$ be the remaining solutions in the iteration.
    In the experiments, this paper considers the following three greedy selection criteria:
    \begin{enumerate}[label=(\alph*)]
        \item \label{sel:rem-min} Remove $\displaystyle \argmin_{S_i \in \tilde{C}} \min_{j \neq i} d(S_i, S_j) $
        \item \label{sel:rem-min-avg} Remove $\displaystyle \argmin_{S_i \in \tilde{C}} \sum_{S_j \in \tilde{S}} d(S_i, S_j) $
        \item \label{sel:rem-sp} Remove $\displaystyle \argmax_{S_i \in \tilde{C}} D_{SP}(\tilde{S}\setminus \{S_i\}) $,
    \end{enumerate}
    where $D_{SP}$ or the Solow-Polasky diversity \cite{Solow1994} is defined as below:
    \begin{dfn}
        Given a constant $\beta \in \xR_{++}$.
        Let ${\mathcal X}$ be a set, $n\in \xN$, $S=(S_1, S_2, \dots, S_{n}) \in \mathcal{X}^n$, and $d: {\mathcal X} \times {\mathcal X} \to \xR_{+} $ be a distance function.
        Then, the Solow-Polasky diversity $D_{SP}:2^{\mathcal X}\to \xR$ is defined by $D_{SP}(S) = {\bm 1}^\trans (D^\prime)^{-1} {\bm 1}$, where ${\bm 1} = (1, 1, \dots, 1)^\trans \in \xR^{n}$ and $D^\prime = (e^{-\beta d(S_i,S_j)})_{ij}$.
    \end{dfn}
   \ref{sel:rem-min} removes the solution with the minimum distance, \ref{sel:rem-min-avg} eliminates the solution with the minimum average distance to all the other remaining solutions, and \ref{sel:rem-sp} removes the solution that minimizes the decrease in the Solow-Polasky diversity of all the remaining solutions in the selected cluster. Ties are broken randomly.

    Note that Ulrich and Thiele \cite{Ulrich2011} give an algorithm to select arbitrary $k\le n$ from $n$ solutions by greedy eliminating solutions with Solow-Polasky diversity as the selection function in $O(n^3)$-time. Using this, we can also compute lines \ref{line:sel-start}--\ref{line:sel-end} in Algorithm \ref{alg:selection} with selection criterion \ref{sel:rem-sp} in $O(n^3)$-time.

    In summary, SCS is described in Algorithm \ref{alg:scs}.

    \begin{algorithm}[H]
        \caption{Sample-Cluster-Select Algorithm}\label{alg:scs}
        \begin{algorithmic}[1]
            \Require $n\in \xN$, $\mathcal{G}$, $d:\mathcal{G}\times \mathcal{G}\to [0,1]$, $0<\delta<1$, $\beta \in \xR_{++}$, a sampling method from $\mathcal{G}$, parameters for the micro-clustering algorithm, a selection criterion
            \Ensure $S_{SEL} \subseteq \mathcal{G}$
            \State {Sample $n$ solutions $S=(S_1, S_2, \dots, S_{n})$ by the sampling method}
            \State {Cluster $S$ into $C=(C_1, C_2, \dots, C_{n_c}) \in (2^S)^{n_c}$}
            \State {Select $S_{SEL} \subseteq S$ from $C$ by Algorithm \ref{alg:selection}}
            \Return $S_{SEL}$
        \end{algorithmic}
    \end{algorithm}

    \section{Problem and local search algorithm}
    \label{sec:probls}
    This section describes the combinatorial optimization problems and local search algorithms used in the experiments.
    The tested problem includes the shortest path problem, the traveling salesman problem, and the set packing problem.
    \subsection{Shortest path problem}

    Let $G=(V, E, w)$ be an undirected edge-weighted graph and given two vertices $s, t \in V$.
    Suppose each edge $e \in E$ has a positive weight $w(e) > 0$.
    Shortest Path problem (SP) aims to find an $s-t$ path with the minimum total weight of its edges.
    This paper assumes that the graph is simple and that no edge has zero weight.
    Since fast algorithms, such as Dijkstra's algorithm \cite{Dijkstra1959}, can solve this problem, no local search algorithm has been proposed.
    Therefore, we establish a new local search (based) algorithm for SP with positive weights.
    Define $\bar{E} = (V\times V) \setminus E$, $\mathcal{P}_G$ by the set of all paths on $G$, and $|p|$ by the length of a path $p$ or the sum of edge weights contained in $p$.
    Given a constant $\lambda > 1$.
    The overview of the proposed algorithm is as follows:
    \begin{enumerate}[label=Step \arabic*., ref=\arabic*]
        \item Compute the all-pairs shortest distances $d_{SD}: V \times V \to \xR_{+}$ and paths $P: V \times V \to \mathcal{P}_G$ on $G$.

        \item Construct a complete weighted graph $G^\prime=(V, V\times V, w^\prime)$ using $d_{SD}$.

        \item Run a local search-based algorithm for the relaxed SP problem on $G^\prime$ and get a path.

    \end{enumerate}
    This algorithm first computes the all-pairs shortest distances $d_{SD}: V \times V \to \xR_{+}$ on $G$ by Floyd-Warshall algorithm \cite{Floyd1962, Warshall1962}.
    Step 2 constructs a complete weighted graph $G^\prime=(V, V\times V, w^\prime)$ with $w^\prime(e) = \lambda d_{SD}(e)$ for all $e \in \bar{E}$ and $w^\prime(e) = w(e)$ otherwise.
    In Step 3, we define the neighborhood $N_{SP}(p)$ of an $s-t$ path $p = sv_1  \dots v_{k-1}t$ on $G^\prime$ at first.
    \begin{dfn}
        Let $p = sv_1 \dots v_{k-1}t$ be an arbitrary $s-t$ path on $G^\prime$, $v_0=s$, and $v_k=t$.
        Then we define $N_{SP}(p) = N_{SP}^1(p) \cup N_{SP}^2(p) \cup N_{SP}^3(p)$, where

        $N_{SP}^1(p) = \left\{sv_1 \dots v_{l-1}v_{l+1} \dots v_{k-1}t : 1 \le l \le k-1  \right\}$,

        $N_{SP}^2(p) = \left\{sv_1 \dots v_{l-1}uv_{l+1} \dots v_{k-1}t : 1 \le l \le k-1, u\in V \right\}$,

        and

        $N_{SP}^3(p) = \left\{sv_1 \dots v_{l-1}uv_{l} \dots v_{k-1}t : 1 \le l \le k, u\in V \right\}$.

    \end{dfn}
    The neighborhood $N_{SP}(p)$ consists of three types of $s-t$ paths: $N_{SP}^1(p)$ is a neighbor obtained by removing a vertex from $p$, $N_{SP}^2(p)$ is obtained by replacing a vertex in $p$ with another vertex, and $N_{SP}^3(p)$ is a neighbor obtained by inserting a vertex into $p$.
    Note that $N_{SP}(p)$ may contain trails even if $p$ is simple.

    Then we show the local search-based algorithm and its subroutine in Algorithms \ref{alg:initial-solution}--\ref{alg:local-search}.
    \begin{algorithm}[H]
        \caption{Algorithm for generating initial solutions}\label{alg:initial-solution}
        \begin{algorithmic}[1]
            \Require $G=(V, E, w)$, $s, t \in V, d_{SD} : V \times V \to \xR_{+}$
            \Ensure $p \in \mathcal{P}_{G}$
            \State $k=0, v_0 = s$
            \While {$v_k \neq t$}
            \State Choose a node $v_{k+1}$ uniformly random from $V_k = \{v : (v_k, v)\in E, d_{SD}(v_k, t) > d_{SD}(v, t)\}$
            \State $k \gets k+1$
            \EndWhile
            \Return $p = sv_1v_2 \dots v_{k-1}t$
        \end{algorithmic}
    \end{algorithm}
    \begin{algorithm}[H]
        \caption{Local search-based algorithm for SP}\label{alg:local-search}
        \begin{algorithmic}[1]
            \Require $G^\prime=(V, V\times V, w^\prime)$, $s, t \in V$
            \Ensure $p \in \mathcal{P}_{G^\prime}$
            \State Calculate $p$ by Algorithm \ref{alg:initial-solution}
            \While {$p$ is not a local optimal path on $G^\prime$}
            \State Pick $q \in N_{SP}(p)$ such that $|p| - |q| > 0$ uniformly random
            \If {$q$ is not simple}
            \State Remove a cycle from $q$
            \EndIf
            \State $p \gets q$
            \EndWhile
            \Return $p$
        \end{algorithmic}
    \end{algorithm}
    Algorithm \ref{alg:initial-solution} generates an initial solution $p$ by choosing a neighbor node of the current node randomly, and Algorithm \ref{alg:local-search} is a local search-based algorithm for SP on $G^\prime$.
    The only difference between our algorithm and regular local search is that we modify it to a simple path by removing the cycle when updating the current solution to a trail.
    This difference is necessary to ensure the algorithm outputs a simple path on $G$.

    Finally, we prove a claim about Algorithm \ref{alg:local-search}.
    \begin{cla}
        \label{cla:simple}
        Algorithm \ref{alg:local-search} always outputs an $s-t$ path on $G$.
    \end{cla}
    \begin{proof}
        Let $p^\prime = sv_1v_2 \dots v_{k-1}t \in \mathcal{P}_{G^\prime}$ be an arbitrary output of the algorithm, $v_0=s$, and $v_k=t$. To prove the claim, it suffices to show that $(v_{i-1}, v_i) \in E$ for all $0 \le i \le k$.
        Thus, for the sake of contradiction, we assume that there exists an integer $1 \le m \le k$ such that $(v_{m-1}, v_m) \in \bar{E}$.
        Let $p_m$ be a shortest path on $G$ between $v_{m-1}$ and $v_m$, and we assume $(v_{m-1}, v^\prime_m) \in p_m$.
        From the assumptions, we have $v^\prime_m \neq v_m$.
        Now we consider a neighbor path $\bar{p} = sv_1v_2 \dots v_{m-1}v^\prime_mv_m \dots v_{k-1}t \in N_{SP}(p^\prime)$.
        We have
        \begin{align}
            |p^\prime| - |\bar{p}| &= w(v_{m-1}, v_m) - w(v_{m-1}, v^\prime_m) - w(v^\prime_m, v_m) \nonumber\\
            &\ge \lambda |p_m| - w(v_{m-1}, v^\prime_m) - \lambda (|p_m| - w(v_{m-1}, v^\prime_m))\nonumber\\
            &= (\lambda-1) w(v_{m-1}, v^\prime_m)\nonumber\\
            &> 0.\label{eq:dif-of-insert-path}
        \end{align}
        This contradicts the assumption that $p$ is a local optimal path on $G^\prime$.
    \end{proof}

    Claim \ref{cla:simple} ensures that Algorithm \ref{alg:local-search} can be used for SP on $G$.

    \subsection{Traveling salesman problem}
    Let $G=(V, E, w)$ be a complete weighted undirected graph.
    Each edge $e \in E = V \times V $ has a non-negative weight $w(e) \ge 0$.
    Traveling Salesman Problem (TSP) aims to find a Hamiltonian cycle with the minimum total weight of edges in the cycle.

    In the experiment, we use random permutations of $V$ as initial solutions and apply the local search algorithm with the 2-OPT neighborhood.

    \subsection{Set packing problem}
    Let $U$ be a finite set, $\mathfrak{S}$ be a family of sets over $U$, and $c : \mathfrak{S}\to \xR_{+}$ be a cost function.
    Define $\mathcal{S} \subseteq \mathfrak{S}$ is a packing of $\mathfrak{S}$ if $s_i \cap s_j = \emptyset$ for all $s_i, s_j \in \mathcal{S}$.
    Set Packing Problem (SPP) aims to find a packing $\mathcal{S} \subseteq \mathfrak{S}$ that maximizes $\sum_{s_i \in \mathcal{S}} c(s_i)$.

    This paper considers rectangle SPP in the experiment. In this problem, vertices on an $n_1 \times n_2$ grid graph mean $U$, and $\mathfrak{S}$ is a set of vertices on an axis-aligned rectangle of the grid.
    We formally define the rectangle set as follows:
    \begin{dfn}
        \label{dfn:rect}
        Let $P = \{1, 2, \dots, n_1\}\times \{1, 2, \dots, n_2\}$ be the positions of the vertices on $n_1 \times n_2$ grid graph, $U = \{1, 2, \dots, n_1\times n_2\}$ be the index set of the vertices, and $X: U \to P$ be the bijection that expresses the position of each vertex.
        Then we define a set $S\subseteq U$ is a rectangle set of $P$ if there exist $1 \le a_1 \le b_1 \le n_1$ and $1 \le a_2 \le b_2 \le n_2$ such that $S = \{i \in U : X(i) \in [a_1, b_1] \times [a_2, b_2]\}$.
    \end{dfn}
    Definition \ref{dfn:rect} defines a rectangle set as a subset whose elements, when mapped to their positions, form vertices of an axis-aligned rectangle.
    Using this definition, we define rectangle SPP as below:
    \begin{dfn}
        \label{dfn:rectSPP}
        Rectangle set packing problem is a set packing problem where all given sets are rectangle sets on an $n_1 \times n_2$ grid graph.
    \end{dfn}
    Since the set takes the form of a rectangle, solutions of rectangle SPP can be visualized as a packing of rectangles.

    Next, we define a neighborhood $N_{SPP}(\mathcal{S})$ which uses a local search algorithm in the experiment:
    \begin{dfn}
        Let $\mathcal{S} \subseteq \mathfrak{S}$ be an arbitrary packing.
        Then we define $N_{SPP}(\mathcal{S}) = \{\mathcal{S}^\prime \in N_{SPP}^1(\mathcal{S}) \cup N_{SPP}^2(\mathcal{S}) : \mathcal{S}^\prime \text{ is a set packing}\}$, where

        $N_{SPP}^1(\mathcal{S}) = \left\{\mathcal{S}^\prime = (\mathcal{S} \setminus \{s_i\}) \cup \{s_j\} : s_i, s_j \in \mathcal{S} \right\}$

        and

        $N_{SPP}^2(\mathcal{S}) = \left\{\mathcal{S}^\prime = \mathcal{S} \cup \{s_i\} : s_i \in \mathcal{S} \right\}$.
    \end{dfn}
    The neighborhood $N_{SPP}(\mathcal{S})$ consists of two types of packings: $N_{SPP}^1(\mathcal{S})$ is a neighbor obtained by removing a set and adding another set from $\mathcal{S}$, and $N_{SPP}^2(\mathcal{S})$ is a neighbor obtained by only adding a set into $\mathcal{S}$.

    \section{Experiments}
    \label{sec:exp}
    This section describes the experimental results comparing the proposed algorithm with others.
    The experiments aim to investigate the performance of SCS in comparison with other algorithms.
    For the dissimilarity function, we employed different measures depending on the problem: the Jaccard distance between edges in the paths or cycles for SP and TSP, and the Jaccard distance between solutions for SPP.
    For any two sets, say $A$ and $B$, the Jaccard distance is defined as $1 - \frac{|A \cap B|}{|A \cup B|}$.

    This paper evaluates the performance of the proposed algorithm by comparing several diversity measures and computation time with other algorithms.
    We employ three diversity measures to evaluate the solution sets:
    \begin{enumerate}
        \item Minimum distance between solutions ($D_{min}$)
        \item Average distance between solutions ($D_{avg}$)
        \item Solow-Polasky diversity ($D_{SP}$)
    \end{enumerate}

    In the experiments, we used Microsoft Surface Laptop Studio 2 Z2F-00018(Corei7-13800H, 2.90GHz, Windows 11 Pro) and ran the experiments on a single node.
    The software used for the computation was VS Community 2022 64bit 17.11.5, Microsoft Visual C++ 2022, and Eigen 3.4.0 \cite{eigenweb} library.
    We used Python 3.12.0 with Matplotlib 3.8.2 \cite{Hunter2007} and NetworkX 3.2.1 \cite{SciPyProceedings_11} libraries to draw the obtained solutions.

    \subsection{Benchmark problems}
    The instances used in the experiments were positive-weighted SP with a 2-dimensional grid graph, SP with Euclidean unit disk graph, Euclidean TSP, and rectangle SPP.
    For each problem, we dealt with two problem sizes, and 20 instances were randomly generated for each.
    Table \ref{tab:problem-parameters} denotes the name and property of the instances.

    \begin{table}[H]
        \centering
        \caption{The parameters of instances used in this paper. ``Instance'' denotes the instance name, ``Problem'' indicates the underlying problem of the instance, ``Type'' refers to the variant of the problem, ``Size'' represents the problem size, and ``$\#$Ins.'' shows the number of instances.}
        \begin{tabular}{|c|c|c|c|c|}
            \hline
            Instance & Problem& Type & Size & $\#$Ins. \\
            \hline
            SP-Grid-100 & SP & Grid & $10\times10$ & 20 \\
            SP-Grid-900 & SP & Grid & $30\times30$ & 20 \\
            SP-EUD-100 & SP & Euclidean Unit Disk & 100 & 20 \\
            SP-EUD-900 & SP & Euclidean Unit Disk & 900 & 20 \\
            TSP-100 & TSP & Euclidean & 100 & 20 \\
            TSP-900 & TSP & Euclidean & 900 & 20 \\
            SPP-100 & SPP & Rectangle & $((10\times10), 100)$ & 20 \\
            SPP-900 & SPP & Rectangle & $((30\times30), 900)$ & 20 \\
            \hline
        \end{tabular}
        \label{tab:problem-parameters}
    \end{table}

    \busub{Shortest Path problem}
    We utilized grid graphs and Euclidean unit disk graphs.
    For the grid graph, we considered two problem sizes: 10$\times$10 and 30$\times$30. The edge weights were uniformly and independently generated in the $[100,000, 1,000,000]$ range.
    For the unit disk graph, we considered two problem sizes: 100 and 900 vertices, and the positions of the vertices were uniformly and independently generated in the range of $[0, 1]^2$. Weights were the Euclidean distance of the vertices, and the disk width was set to $\frac{2}{\sqrt{|V|}}$, which is approximately twice the edge length in a grid graph with the same node number. Distances were rounded to the sixth decimal place.
    We generated the graphs using this method and only employed the connected graphs as benchmark problems.

    \busub{Traveling salesman problem}
    This paper experimented Euclidean traveling salesman problems.
    We considered two problem sizes: 100 and 900 vertices. The positions of the vertices were uniformly and independently generated in the range of $[0.0, 1.0]^2$. We did not allow duplicate vertex positions. The positions and distances between the vertices were rounded to the sixth decimal place.

    \busub{Set packing problem}
    This paper employed rectangle set packing problems.
    We considered two problem sizes: 10$\times$10 and 30$\times$30, and the numbers of sets were 100 and 900, respectively. The rectangles of the sets were uniformly and independently selected in the range of $[1.0, 5.0]$ for both the width and height, and then, a feasible one of the four rotations was randomly chosen. The weight of the set $w(S)$ was determined by rounding $r|S|$ to the nearest integer, where $r$ is a random number uniformly selected from $[1.0, 5.0]$. We ensured that the set with the same elements was not generated in the same instance.

    \subsection{Experimental settings}
    This paper describes the parameter settings and the experimental methods for SCS.
    We set the parameters used in SCS: the number of generated solutions $n$ was 1,000, and the proximity threshold of solutions $\delta$ was 0.1. In micro-clustering, we set the minimum clique size to 5, the similarity threshold to 0.15, the maximum repetition number to 99, and the similarity measure to the resemblance function. In the local search algorithm for SP, we set $\lambda$ to 1.1. The parameter $\beta$ in the Solow-Polasky diversity was $2.0 \times 10^{-5}$.

    Using these parameters, we tested SCS with three different selection criteria: \ref{sel:rem-min}, \ref{sel:rem-min-avg}, and \ref{sel:rem-sp}.
    For each experiment, the results up to the setup of the selection step were shared among these three variants, while only the final selection step differed.

    \subsection{Compared algorithms}
    The compared algorithms were simple multi-start local search, Yen's algorithm \cite{Yen1970}, NOAH \cite{Ulrich2011}, and SCS. Yen's algorithm was applied only to SP, and NOAH was tested only on TSP and SPP.
    The multi-start local search algorithm outputs the first $n_s$ distinct solutions in the sampling step of SCS, where $n_s$ represents the number of the solutions output by SCS. These solutions were also utilized as the initial solution of NOAH.
    In Yen's algorithm, we applied Local search to enumerate local optimal solutions.
    The algorithm generated $n_s$ different solutions or terminated after 100,000 iterations.
    The original Yen's algorithm was also used as a reference without considering the local optimal condition.
    In the NOAH experiments, we considered three diversity functions and two termination conditions. Therefore, we experimented with six NOAH variants.
    The diversity functions were the experimented diversity measures ($D_{min}$, $D_{avg}$, and $D_{SP}$), and the termination conditions were the execution time limit and the number of consecutive generations with no improvements.
    We set the time limit as the maximum execution time of three SCS results and the number of generations to 20.
    Since our problem does not consider the objective values of each solution, we applied the same local search algorithm as used by SCS instead of another EA in its solution improvement phase of NOAH.
    Thus, it became a Memetic algorithm rather than a two-phase EA.
    In the TSP experiments of NOAH, we used PMX crossover and swap mutation.
    For SPP, NOAH employed a crossover used in Delorme et al. \cite{Delorme2010} and a mutation operator that randomly removes a single set packed by the solution.
    Unless otherwise stated, the parameters of NOAH followed Ulrich and Thiele \cite{Ulrich2011}.
    Table \ref{tab:comparing-algorithm} summarizes and names the comparing algorithms.
    \begin{table}[H]
        \centering
        \caption{Comparison of algorithm properties. ``Type'' denotes the type of the algorithm, ``Diversity'' indicates the diversity measure used in the algorithm, and ``Termination'' describes the termination condition of the algorithm. }
        {
            \begin{tabular}{|c|c|c|c|c|}
                \hline
                Algorithm & Type& Diversity & Termination \\
                \hline
                MSLS & LS & - & Find $n_s$ distinct solutions or 100,000 iterations\\
                Yen & $k$-best & - & Find $n_s$ distinct solutions or 100,000 iterations\\
                Yen-LS & $k$-best & - & Find $n_s$ distinct solutions or 100,000 iterations\\
                NOAH-MIN-O & NOAH & $D_{min}$ & No improvement for 20 generations. \\
                NOAH-MIN-T & NOAH & $D_{min}$ & Time limit (Max time among SCS results) \\
                NOAH-AVG-O & NOAH & $D_{avg}$ & No improvement for 20 generations. \\
                NOAH-AVG-T & NOAH & $D_{avg}$ & Time limit (Max time among SCS results) \\
                NOAH-SP-O & NOAH & $D_{SP}$ & No improvement for 20 generations. \\
                NOAH-SP-T & NOAH & $D_{SP}$ & Time limit (Max time among SCS results)\\
                SCS-MIN & SCS & $D_{min}$ & -  \\
                SCS-AVG & SCS & $D_{avg}$ & -  \\
                SCS-SP & SCS & $D_{SP}$ & -  \\
                \hline
            \end{tabular}
        }
        \label{tab:comparing-algorithm}
    \end{table}

    \subsection{Results}
    \label{sec:results}
    The experiment results are shown in Table \ref{tab:results}.
\begin{table}[H]
    \centering
    \footnotesize
    \caption{Performance comparison results for all tested instances. Values show average improvement ratios [\%] of diversity measures ($D_{min}$, $D_{avg}$, $D_{SP}$), average number of output solutions ($n_s$), and average computation time compared to multi-start local search (MSLS) across 20 instances per problem type. Parenthetical values in MSLS rows represent absolute average values of each indicator. Best results for each instance are shown in bold.}
    \label{tab:results}
    \scalebox{0.75}
    {
        \begin{tabular}{|l|rrrrr|rrrrr|}
            \hline
            & \multicolumn{5}{c|}{SP-Grid-100} & \multicolumn{5}{c|}{SP-Grid-900} \\
            \hline
            Algorithm & $D_{min}[\%]$ & $D_{avg}[\%]$ & $D_{SP}[\%]$ & $n_s$ & time[s] & $D_{min}[\%]$ & $D_{avg}[\%]$ & $D_{SP}[\%]$ & $n_s$ & time[s] \\\hline
            MSLS & $0\%(0.268)$ & $0\%(0.589)$ & $0\%(6.86)$ & 7.25 & 0.000931 & $0\%(0.205)$ & $0\%(0.635)$ & $0\%(17.0)$ & 18.8 & 0.0585 \\
            \hline
            Yen & $-22.6\%$ & $-29.6\%$ & $-10.7\%$ & 7.25 & 0.0299 & $-55.0\%$ & $-58.9\%$ & $-60.0\%$ & 18.8 & 1.08 \\
            Yen-LS & $-7.43\%$ & $-12.3\%$ & $-2.20\%$ & 7.25 & 0.0150 & $-46.1\%$ & $-42.6\%$ & $-45.7\%$ & 18.8 & 4.56 \\\hline
            SCS-AVG & $60.5\%$ & $20.2\%$ & $3.45\%$ & 7.25 & 2.43 & $204\%$ & $\bm{23.2\%}$ & $9.15\%$ & 18.8 & 6.26 \\
            SCS-MIN & $58.5\%$ & $14.3\%$ & $3.25\%$ & 7.25 & 2.43 & $243\%$ & $18.7\%$ & $9.02\%$ & 18.8 & 6.40 \\
            SCS-SP & $\bm{80.7\%}$ & $\bm{20.5\%}$ & $\bm{3.59\%}$ & 7.25 & 2.43 & $\bm{265\%}$ & $21.9\%$ & $\bm{9.40\%}$ & 18.8 & 7.38 \\
            \hline
            & \multicolumn{5}{c|}{SP-EUD-100} & \multicolumn{5}{c|}{SP-EUD-900} \\ \hline
            MSLS & $0\%(0.438)$ & $0\%(0.788)$ & $0\%(7.35)$ & 7.40 & 0.000885 & $0\%(0.483)$ & $0\%(0.878)$ & $0\%(20.6)$ & 20.7 & 0.115 \\
            \hline
            Yen & $-44.2\%$ & $-46.3\%$ & $-13.0\%$ & 7.40 & 0.0172 & $-78.3\%$ & $-74.0\%$ & $-68.3\%$ & 20.7 & 0.608 \\
            Yen-LS & $-1.54\%$ & $-12.7\%$ & $-5.62\%$ & 7.30 & 67.2 & $-55.4\%$ & $-56.0\%$ & $-55.6\%$ & 13.6 & 472 \\\hline
            SCS-AVG & $51.2\%$ & $\bm{12.6\%}$ & $0.620\%$ & 7.40 & 1.76 & $74.4\%$ & $\bm{6.93\%}$ & $0.406\%$ & 20.7 & 8.61 \\
            SCS-MIN & $47.8\%$ & $10.1\%$ & $0.598\%$ & 7.40 & 1.76 & $79.9\%$ & $5.48\%$ & $0.393\%$ & 20.7 & 8.82 \\
            SCS-SP & $\bm{58.9\%}$ & $12.2\%$ & $\bm{0.627\%}$ & 7.40 & 1.76 & $\bm{82.4\%}$ & $6.73\%$ & $\bm{0.409\%}$ & 20.7 & 10.3 \\
            \hline
             & \multicolumn{5}{c|}{TSP-100} & \multicolumn{5}{c|}{TSP-900} \\ \hline
            MSLS & $0\%(0.359)$ & $0\%(0.536)$ & $0\%(28.2)$ & 31.6 & 0.0451 & $0\%(0.550)$ & $0\%(0.591)$ & $0\%(15.3)$ & 15.7 & 2.53 \\
            \hline
            NOAH-AVG-O & $-100\%$ & $\bm{28.7\%}$ & $-100\%$ & 31.6 & 11.6 & $7.30\%$ & $7.92\%$ & $0.977\%$ & 15.7 & 297 \\
            NOAH-AVG-T & $-33.8\%$ & $27.5\%$ & $-3.18\%$ & 31.6 & 7.12 & $2.57\%$ & $6.88\%$ & $0.852\%$ & 15.7 & 165 \\
            NOAH-MIN-O & $25.1\%$ & $3.81\%$ & $2.52\%$ & 31.6 & 0.454 & $1.08\%$ & $-0.479\%$ & $-0.0735\%$ & 15.7 & 17.7 \\
            NOAH-MIN-T & $29.3\%$ & $4.84\%$ & $3.03\%$ & 31.6 & 7.11 & $1.79\%$ & $-0.585\%$ & $-0.0868\%$ & 15.7 & 165 \\
            NOAH-SP-O & $53.0\%$ & $25.8\%$ & $\bm{9.04\%}$ & 31.6 & 8.76 & $\bm{10.3\%}$ & $\bm{8.33\%}$ & $\bm{1.02\%}$ & 15.7 & 363 \\
            NOAH-SP-T & $\bm{54.7\%}$ & $25.9\%$ & $8.97\%$ & 31.6 & 7.12 & $8.02\%$ & $6.79\%$ & $0.852\%$ & 15.7 & 165 \\\hline
            SCS-AVG & $27.6\%$ & $11.6\%$ & $5.67\%$ & 31.6 & 5.65 & $3.86\%$ & $2.46\%$ & $0.358\%$ & 15.7 & 165 \\
            SCS-MIN & $33.8\%$ & $8.58\%$ & $4.70\%$ & 31.6 & 5.89 & $4.77\%$ & $1.52\%$ & $0.233\%$ & 15.7 & 165 \\
            SCS-SP & $31.4\%$ & $11.5\%$ & $5.70\%$ & 31.6 & 7.10 & $4.04\%$ & $2.40\%$ & $0.351\%$ & 15.7 & 165 \\
            \hline
            & \multicolumn{5}{c|}{SPP-100} & \multicolumn{5}{c|}{SPP-900} \\ \hline
            MSLS & $0\%(0.264)$ & $0\%(0.753)$ & $0\%(20.6)$ & 21.1 & 0.00221 & $0\%(0.638)$ & $0\%(0.780)$ & $0\%(37.9)$ & 38.2 & 0.597 \\
            \hline
            NOAH-AVG-O & $63.6\%$ & $18.3\%$ & $1.89\%$ & 21.1 & 0.234 & $-100\%$ & $2.95\%$ & $-100\%$ & 38.2 & 58.9 \\
            NOAH-AVG-T & $140\%$ & $\bm{21.2\%}$ & $2.37\%$ & 21.1 & 4.81 & $-100\%$ & $2.28\%$ & $-100\%$ & 38.2 & 19.2 \\
            NOAH-MIN-O & $108\%$ & $1.61\%$ & $1.40\%$ & 21.1 & 0.0555 & $0.933\%$ & $-0.337\%$ & $-0.0349\%$ & 38.2 & 2.73 \\
            NOAH-MIN-T & $141\%$ & $5.81\%$ & $1.88\%$ & 21.1 & 4.81 & $1.31\%$ & $-0.824\%$ & $-0.0744\%$ & 38.2 & 19.2 \\
            NOAH-SP-O & $158\%$ & $13.4\%$ & $2.25\%$ & 21.1 & 0.168 & $-3.42\%$ & $1.84\%$ & $0.129\%$ & 38.2 & 40.3 \\
            NOAH-SP-T & $\bm{194\%}$ & $20.0\%$ & $\bm{2.41\%}$ & 21.1 & 4.81 & $-4.49\%$ & $1.39\%$ & $0.0995\%$ & 38.2 & 19.2 \\\hline
            SCS-AVG & $103\%$ & $16.6\%$ & $2.19\%$ & 21.1 & 3.90 & $8.01\%$ & $\bm{3.82\%}$ & $0.265\%$ & 38.2 & 18.4 \\
            SCS-MIN & $139\%$ & $12.4\%$ & $2.19\%$ & 21.1 & 4.03 & $\bm{12.5\%}$ & $2.47\%$ & $0.196\%$ & 38.2 & 18.5 \\
            SCS-SP & $146\%$ & $15.5\%$ & $2.29\%$ & 21.1 & 4.81 & $10.2\%$ & $3.77\%$ & $\bm{0.266\%}$ & 38.2 & 19.0 \\
            \hline
        \end{tabular}
    }
\end{table}

First, we analyze the experimental results for SP.
These results show that the proposed SCS consistently outperformed MSLS and $k$-best algorithms across all problems and evaluation measures.
For EUD instances, Yen-LS occasionally failed to generate $n_s$ solutions.
This fact indicates that even among the top 100,000 solutions ranked by objective value, there were insufficient local optimal solutions.
This finding reveals that sub-optimal solutions lack sufficient diversity, demonstrating the fundamental limitation of the $k$-best strategy in generating diverse solutions for Euclidean SP problems.
Based on these results, SCS is a practical algorithm for generating diverse solution sets in SP problems.

Next, we examine the results for TSP.
Comparing SCS with MSLS, the experimental results show that SCS outperformed MSLS across all diversity measures although the performance gap between SCS and MSLS are small in the bigger problem.
In the comparison between SCS and NOAH, NOAH-SP-O and NOAH-SP-T exhibited better performance than SCS.
NOAH-AVG-O is also superior to SCS without consideration of $D_{min}$.
These findings indicate that while SCS is more effective than simple sampling for TSP instances, it does not achieve the same level of performance as NOAH.
To analyze these results in more detail, we show the histograms of the distributions of the dissimilarities between sampled 1,000 solutions by SCS of the first instances for each experimented problem in Figures \ref{fig:dissim-SP}--\ref{fig:dissim-others}.
\begin{figure}
    \centering
    \subfigure[SP-Grid-100]{\includegraphics[width=0.4\textwidth]{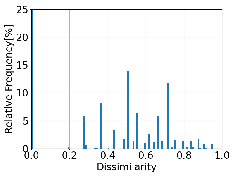}}
    \subfigure[SP-Grid-900]{\includegraphics[width=0.4\textwidth]{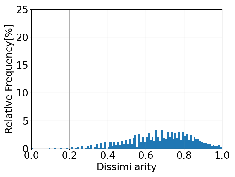}}\\
    \centering
    \subfigure[SP-EUD-100]{\includegraphics[width=0.4\textwidth]{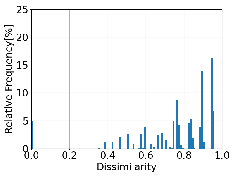}}
    \subfigure[SP-EUD-900]{\includegraphics[width=0.4\textwidth]{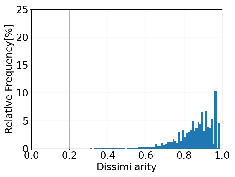}}\\
    \caption{Distributions of pairwise dissimilarities between 1,000 sampled solutions generated by SCS for the first instance of Shortest path problems. Identical solutions are counted if sampled independently (485,000 pairs are counted). Problem names are shown in individual captions. The vertical axis represents relative frequency [\%], while the horizontal axis shows dissimilarity values, with data aggregated in bins of 0.01 width.}
    \label{fig:dissim-SP}
\end{figure}
\begin{figure}
    \centering
    \subfigure[TSP-100]{\includegraphics[width=0.4\textwidth]{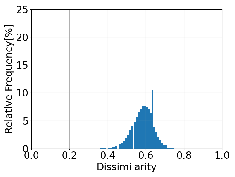}}
    \subfigure[TSP-900]{\includegraphics[width=0.4\textwidth]{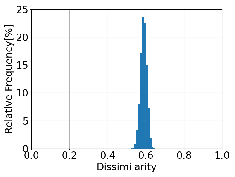}}\\
    \centering
    \subfigure[SPP-100]{\includegraphics[width=0.4\textwidth]{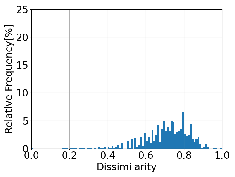}}
    \subfigure[SPP-900]{\includegraphics[width=0.4\textwidth]{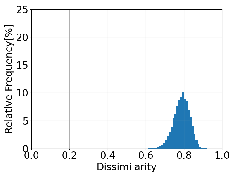}}
    \caption{Distributions of pairwise dissimilarities between 1,000 sampled solutions generated by SCS for the first instance of TSP problems and SPP problems. Identical solutions are counted if sampled independently (485,000 pairs are counted). Problem names are shown in individual captions. The vertical axis represents relative frequency [\%], while the horizontal axis shows dissimilarity values, with data aggregated in bins of 0.01 width.}
    \label{fig:dissim-others}
\end{figure}
Figure \ref{fig:dissim-others} shows that the histograms of TSP have relatively small peak values, approximately equal to 0.6, corresponding to sharing 57\% of edges between distinct solutions on average.
This finding indicates that many edges in the solutions have high probabilities of occurrences and that, on average, there were less than two options for each edge.
Therefore, even though SCS's results were superior to MSLS, NOAH may have more advantages in the TSP instances because it can use its EA operator to search for solutions with low occurrence probabilities in the sampling phase.

Finally, we analyze the results for SPP.
In SPP-100, we can observe trends similar to those in TSP problems.
While SCS performed better than MSLS, it fell short of the improvements relative to NOAH-SP-O and NOAH-SP-T.
However, SCS outperformed both MSLS and NOAH across all diversity measures in SPP-900, although SCS's performance for TSP-900 is worse than TSP-100.
Two factors may explain these differences.
First, the solution space of SPP may be more challenging for NOAH's EA operators to explore.
Second is the difference in local optimal solution distribution between TSP and SPP.
Since the 2-OPT local optimality criterion only evaluates the improvement of edge exchanges in local regions and has little effect on the optimality of structures in distant parts of the tour, TSP's local optimal solutions consist of partial structures.
This characteristic results in relatively small and concentrated dissimilarities between solutions, making it difficult to discover distinct solution patterns through clustering.
In contrast, the local optimal solutions in SPP likely have more distorted distributions, which may be more suitable for clustering-based approaches.
These findings suggest that SCS could be relatively effective in solving computationally hard problems.

Based on these results, we discuss the capabilities of the proposed SCS algorithm.
The experimental findings show that SCS outperformed MSLS across the all evaluation measures, validating the fundamental design of our approach.
In SP problems, SCS demonstrated better performances for all measures over $k$-best algorithms, highlighting its strength in generating diverse solutions for SP.
Although SCS did not match the performance of some NOAH variants in TSP and SPP except for SPP-900, it consistently outperformed MSLS or straightforward enumeration methods.
This observation implies that, while SCS may not compete with advanced heuristics like NOAH, it offers a viable solution for problems where such sophisticated heuristics may underperform and for newly emerging problems that lack specialized algorithms.
This applicability to new problems is particularly valuable as SCS requires only a sampling method or solution neighborhood structure for implementation, making it readily adaptable to various optimization problems where sophisticated heuristics have yet to be developed.

\subsection{Visualization of clustered solutions}
Clustering potential solutions is a critical operation in SCS.
Since similar solutions are grouped within each cluster, aggregating the solutions in each cluster may reveal diverse characteristics of potential solutions, thereby addressing Issues \ref{iss:blackbox}--\ref{iss:suspect}.
Here, we present visualizations of aggregated solutions from the four biggest clusters obtained in our numerical experiments.
Results are shown for the first of 20 instances of each problem.
Figures \ref{fig:SPg-clusters}--\ref{fig:SPP-clusters} show these cluster visualizations.
Elements (edges or sets) are colored based on their sharing ratio within each cluster, with the ratio-to-color correspondence shown in the bottom color bar of each figure.
For visibility, elements are not colored if their sharing ratio is 0 in SP or 0.5 or less in TSP and SPP.
The caption of each superimposition figure indicates the size of the corresponding cluster.

The SP results are shown in Figures \ref{fig:SPg-clusters}--\ref{fig:SPu-clusters}.
\begin{figure}[H]
    \centering
    \footnotesize
    \subfigure[size--617]{\includegraphics[width=0.27\textwidth]{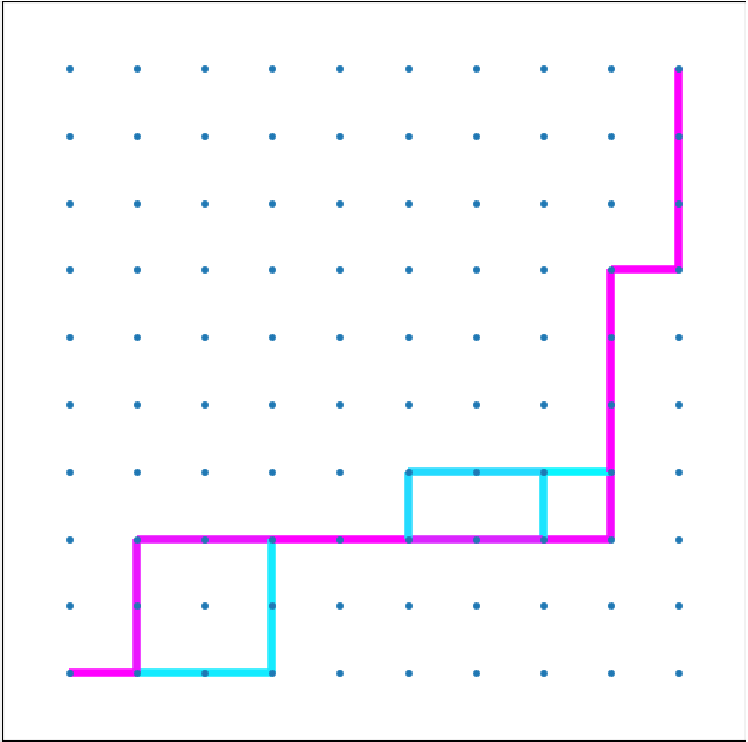}}
    \subfigure[size--157]{\includegraphics[width=0.27\textwidth]{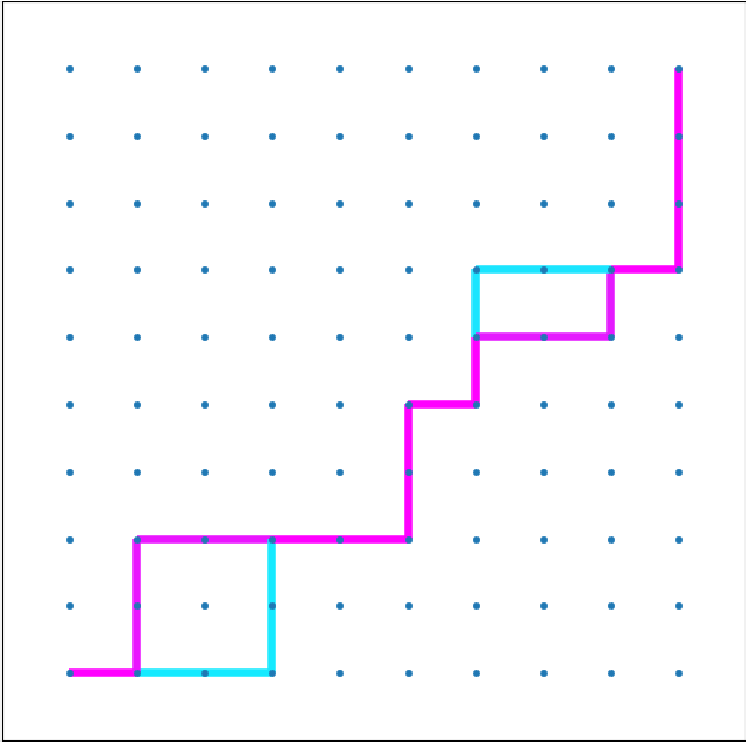}}
    \subfigure[size--152]{\includegraphics[width=0.27\textwidth]{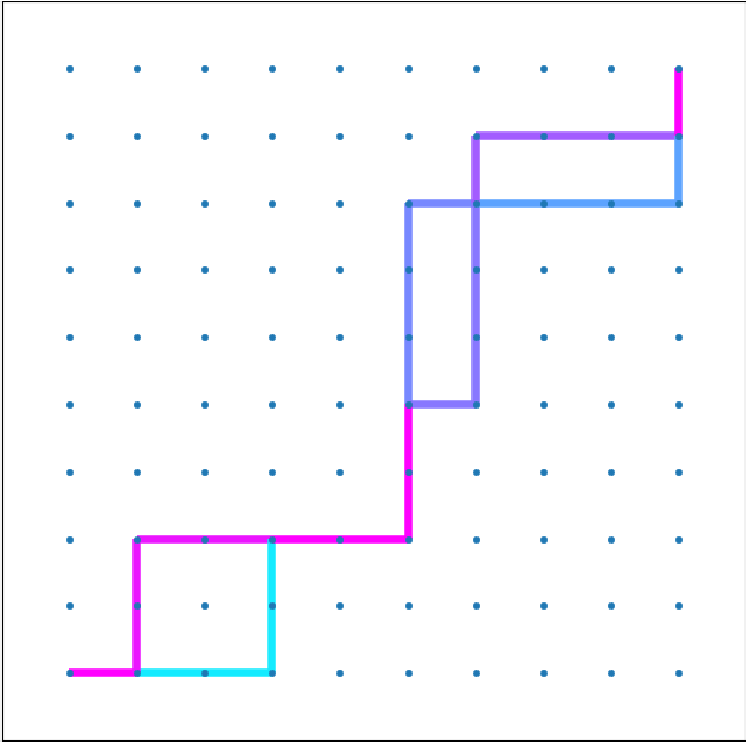}}\\
    \centering
    \subfigure[size--471]{\includegraphics[width=0.27\textwidth]{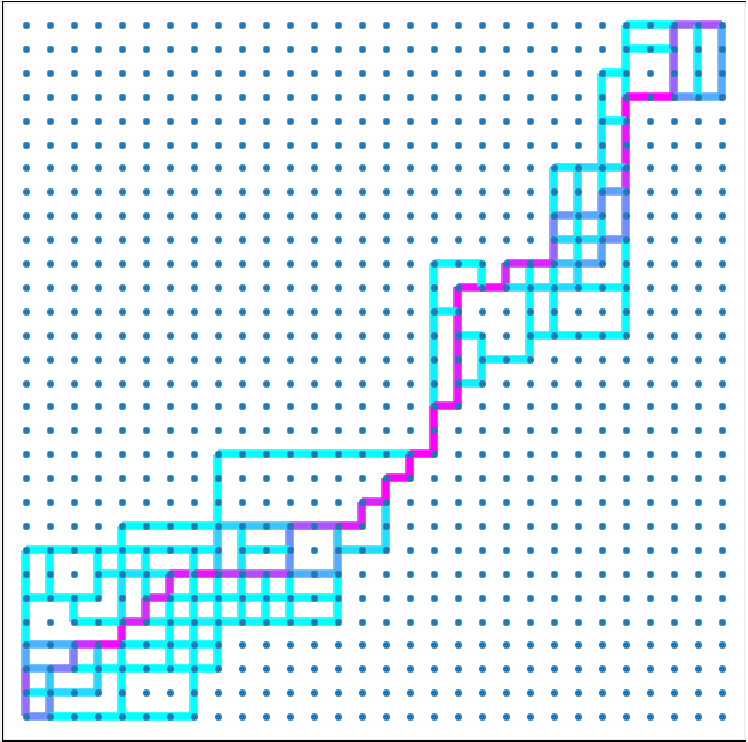}}
    \subfigure[size--101]{\includegraphics[width=0.27\textwidth]{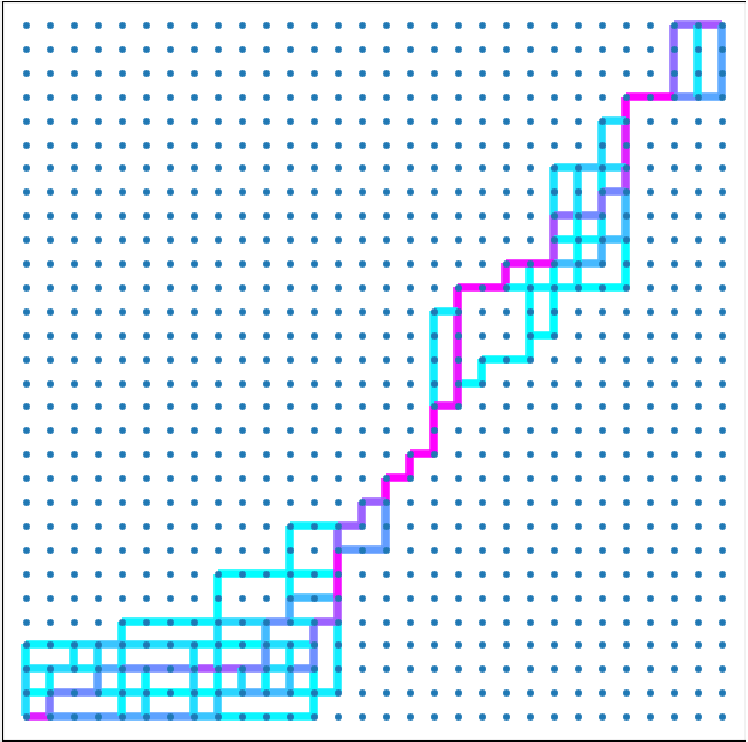}}
    \subfigure[size--63]{\includegraphics[width=0.27\textwidth]{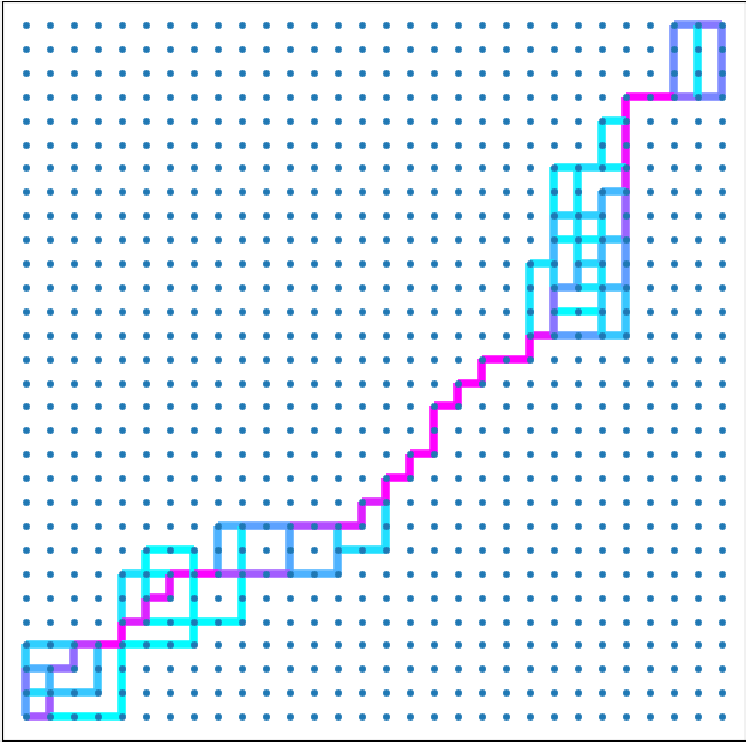}}\\
    \centering
    \subfigure{\includegraphics[width=0.7\textwidth]{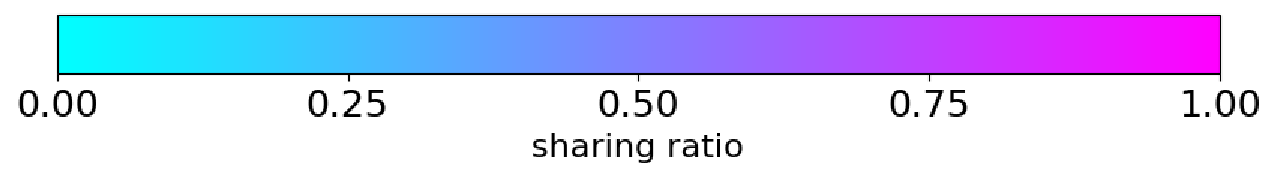}}
    \caption{Superimpositions of three biggest clusters obtained by SCS in the first experimental instances. (a)--(c) show the results of SP-Grid-100 and (d)--(f) depict the results of SP-Grid-900. Blue points represent nodes, with the source node at the bottom left and the sink node at the top right. Numbers in subcaptions indicate cluster sizes. The color bar at the bottom shows the sharing ratio of edges in each cluster.}
    \label{fig:SPg-clusters}
\end{figure}
\begin{figure}[H]
    \centering
    \subfigure[size--462]{\includegraphics[width=0.27\textwidth]{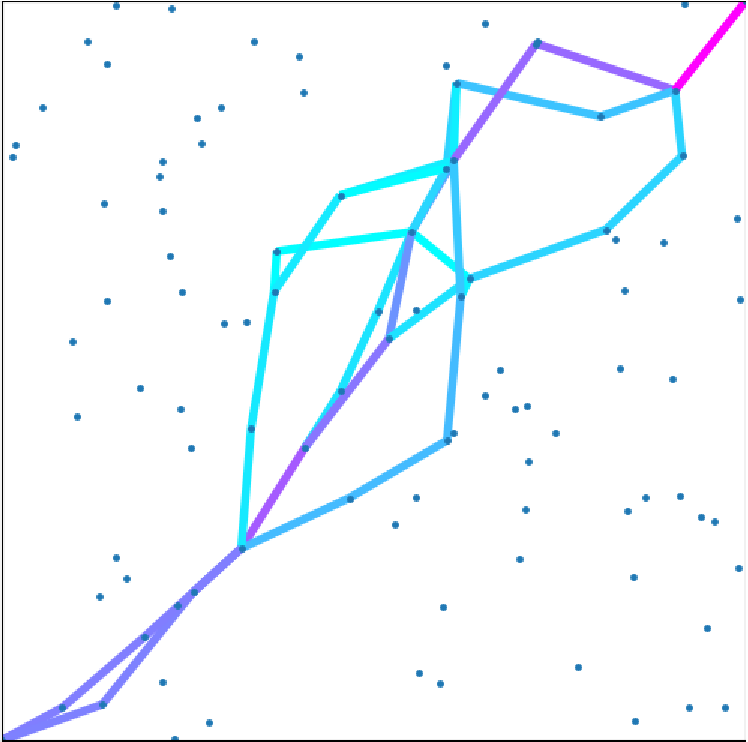}}
    \subfigure[size--228]{\includegraphics[width=0.27\textwidth]{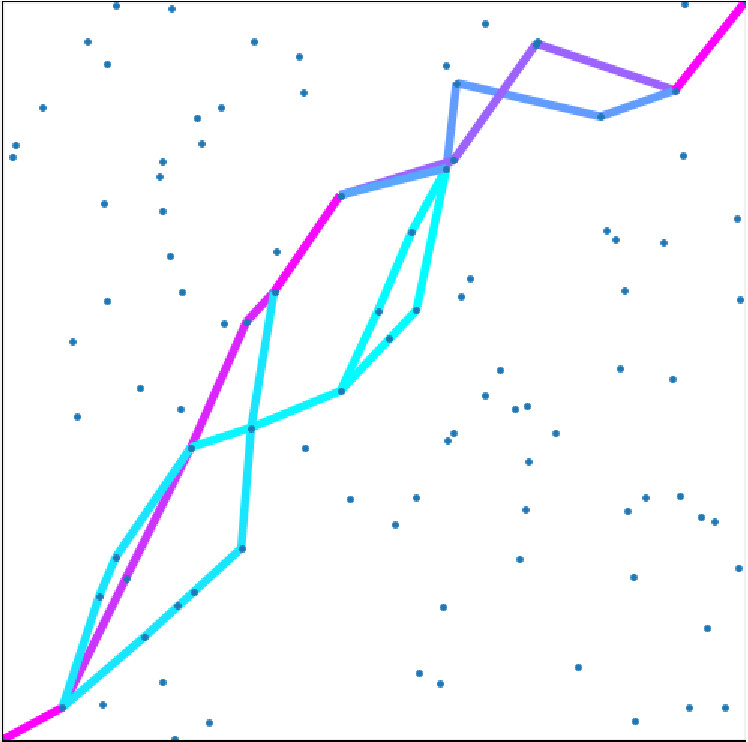}}
    \subfigure[size--158]{\includegraphics[width=0.27\textwidth]{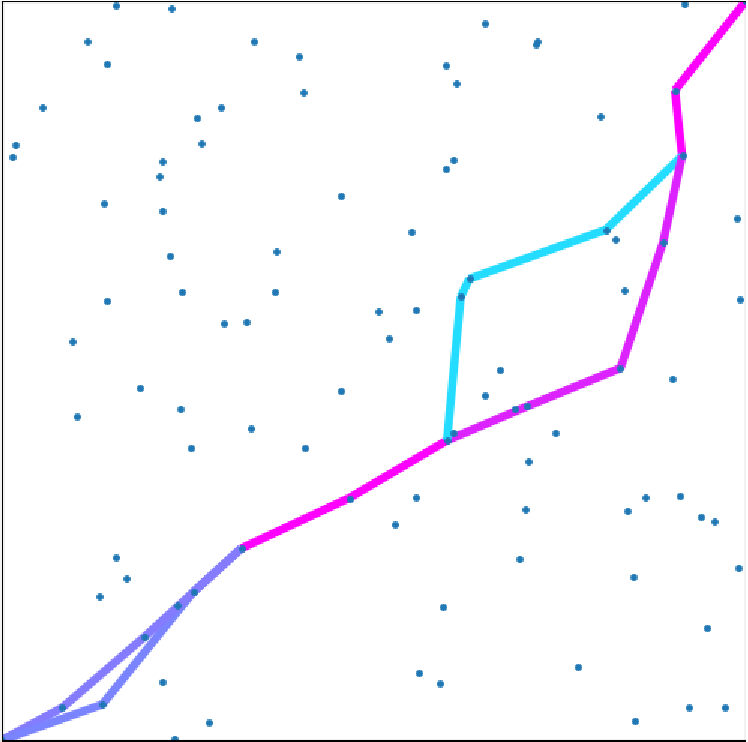}}\\

    \subfigure[size--735]{\includegraphics[width=0.27\textwidth]{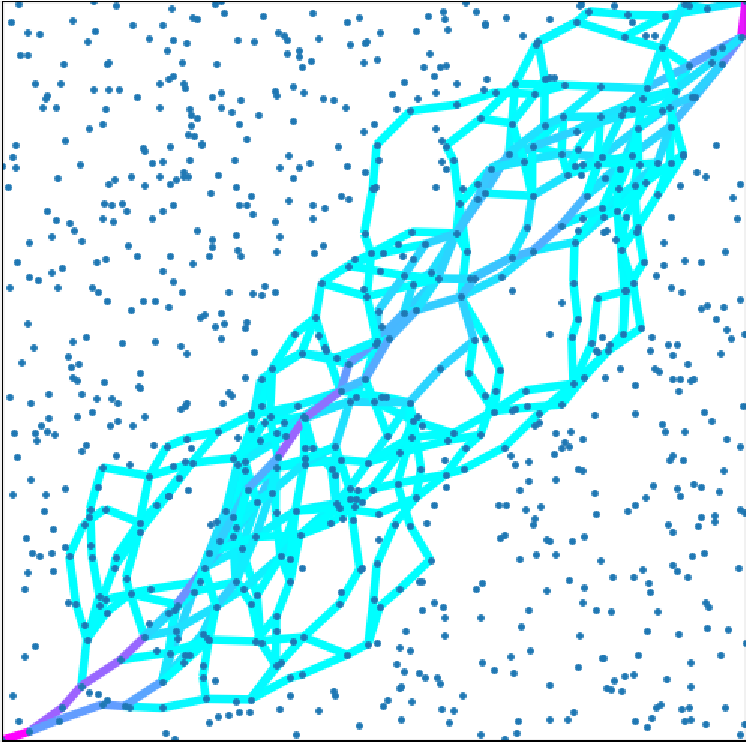}}
    \subfigure[size--74]{\includegraphics[width=0.27\textwidth]{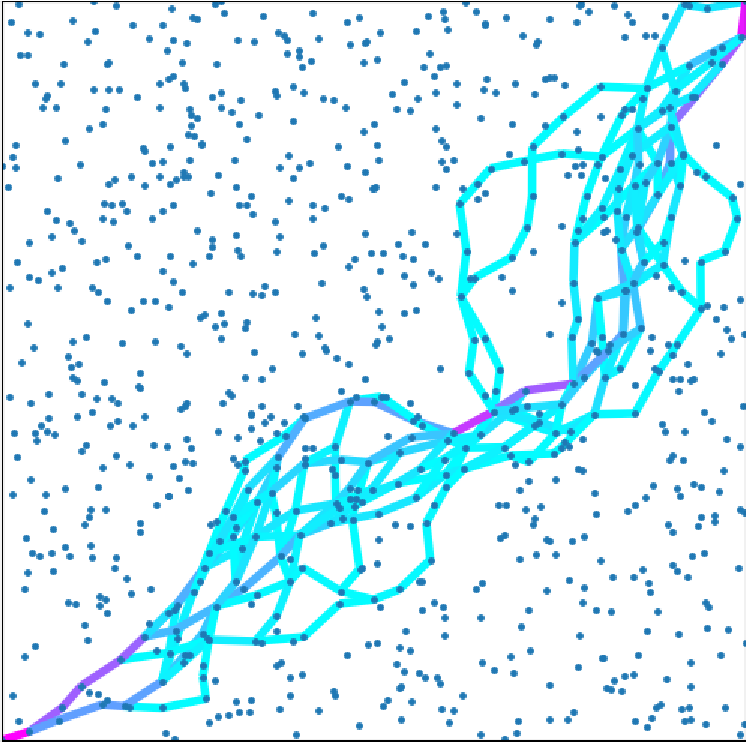}}
    \subfigure[size--33]{\includegraphics[width=0.27\textwidth]{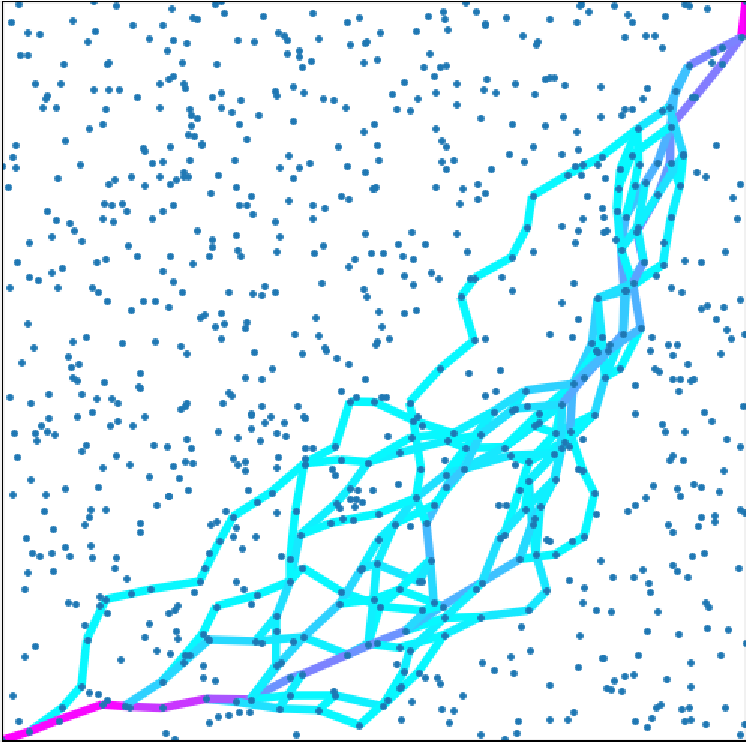}}\\
    \centering
    \subfigure{\includegraphics[width=0.7\textwidth]{cool_colormap_SP.eps}}
    \caption{Superimpositions of three biggest clusters obtained by SCS in the first experimental instances. (a)--(c) show the results of SP-EUD-100 and (d)--(f) depict the results of SP-EUD-900. Blue points represent nodes, with the source node at the bottom left and the sink node at the top right. Numbers in subcaptions indicate cluster sizes. The color bar at the bottom shows the sharing ratio of edges in each cluster.}
    \label{fig:SPu-clusters}
\end{figure}
The SP results are shown in Figures \ref{fig:SPg-clusters}--\ref{fig:SPu-clusters}.
Figure \ref{fig:SPg-clusters} shows distinct central paths, colored red, surrounded by several edges with lower sharing ratios.
Besides, we can distinguish between well-shared and less-shared parts within clusters. For example, in (c), while a single path dominates on the left side, the right regions contain several choices.
In Figure \ref{fig:SPu-clusters}, although the central path in SP-EUD-900 is unclear, similar trends to the grid graph results can be observed.
These visualizations, showing central solutions and their surrounding alternatives, provide more information than presenting individual solutions and can serve as a solution to Issue \ref{iss:not_opt}, especially when there are unmodeled preferences.
Moreover, highly shared parts can be interpreted as crucial components of potential solutions, while less shared parts can be considered less important.
These findings provide important insights for solving optimization problems and address Issue \ref{iss:blackbox}.
Furthermore, visualizing diverse partial structures of possible solutions and their surrounding alternatives provides an overview of potential solutions, contributing to solving Issue \ref{iss:suspect}.
Thus, this clustering approach to SP addresses several issues that existing methods struggle with, as described in Section \ref{sec:intro}.

TSP results are presented in Figure \ref{fig:TSP-clusters}.
\begin{figure}[H]
    \centering
    \footnotesize
    \subfigure[size--551]{\includegraphics[width=0.27\textwidth]{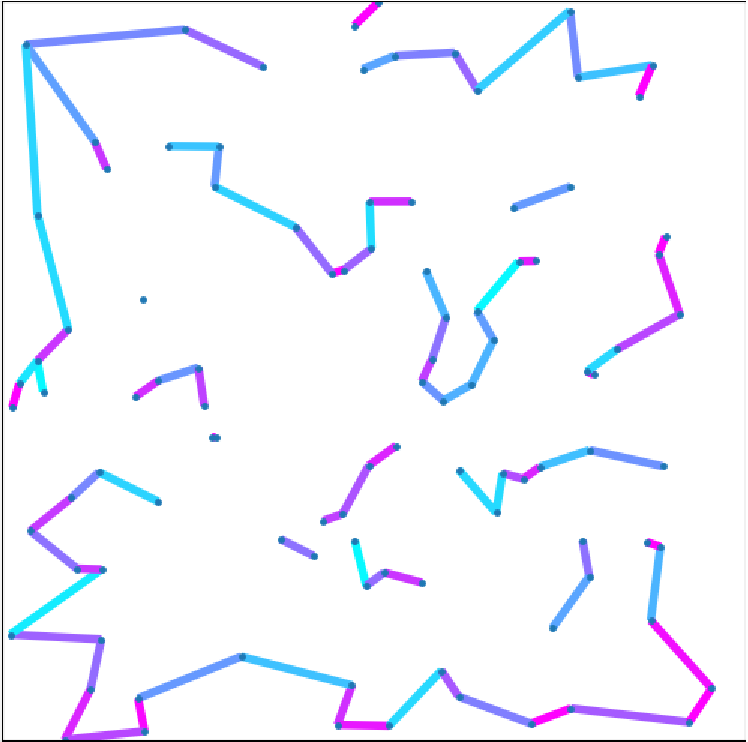}}
    \subfigure[size--44]{\includegraphics[width=0.27\textwidth]{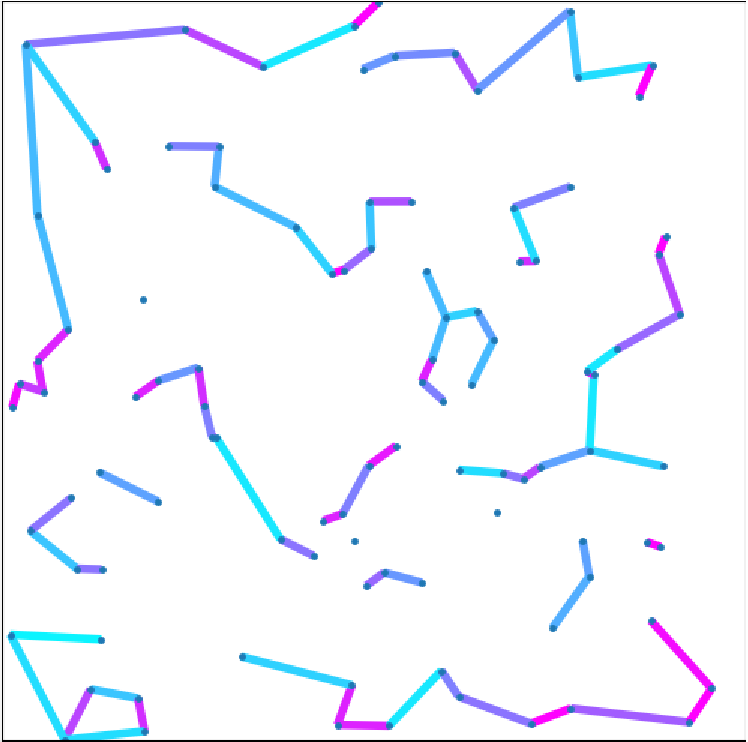}}
    \subfigure[size--25]{\includegraphics[width=0.27\textwidth]{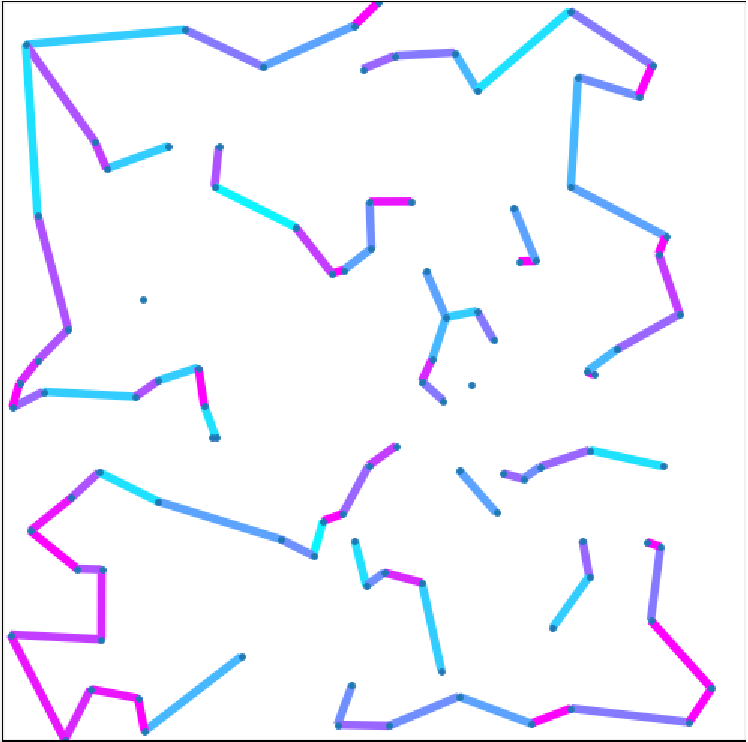}}\\
    \centering
    \subfigure[size--420]{\includegraphics[width=0.27\textwidth]{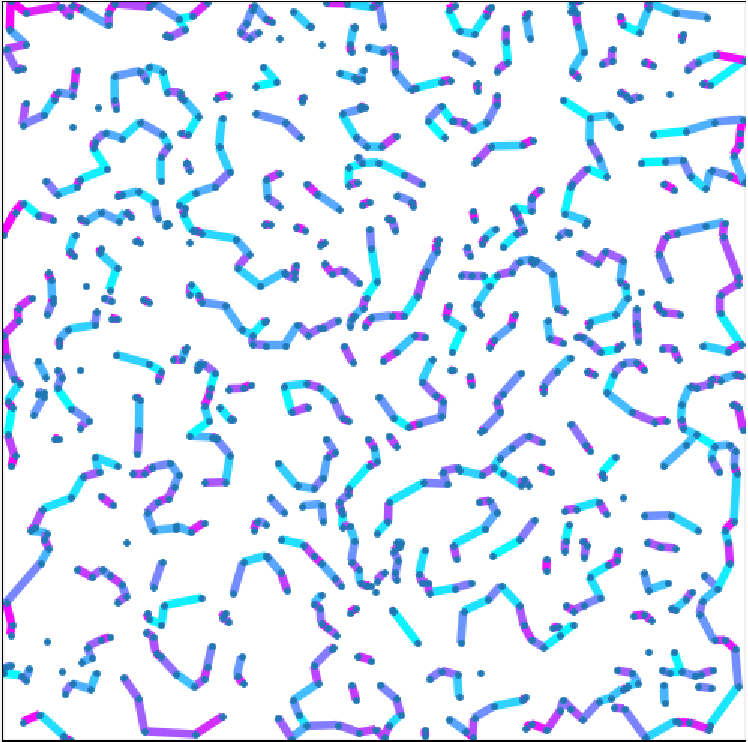}}
    \subfigure[size--13]{\includegraphics[width=0.27\textwidth]{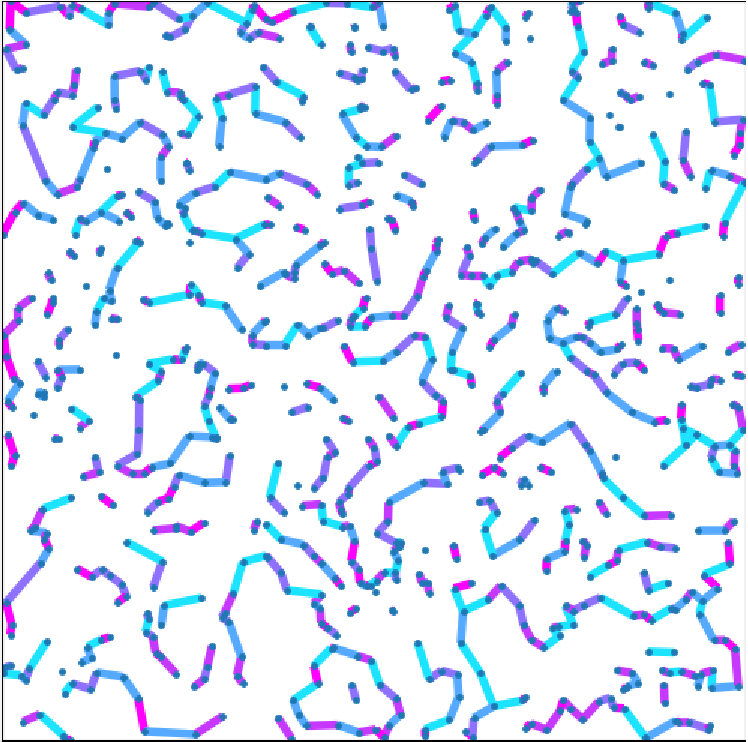}}
    \subfigure[size--7]{\includegraphics[width=0.27\textwidth]{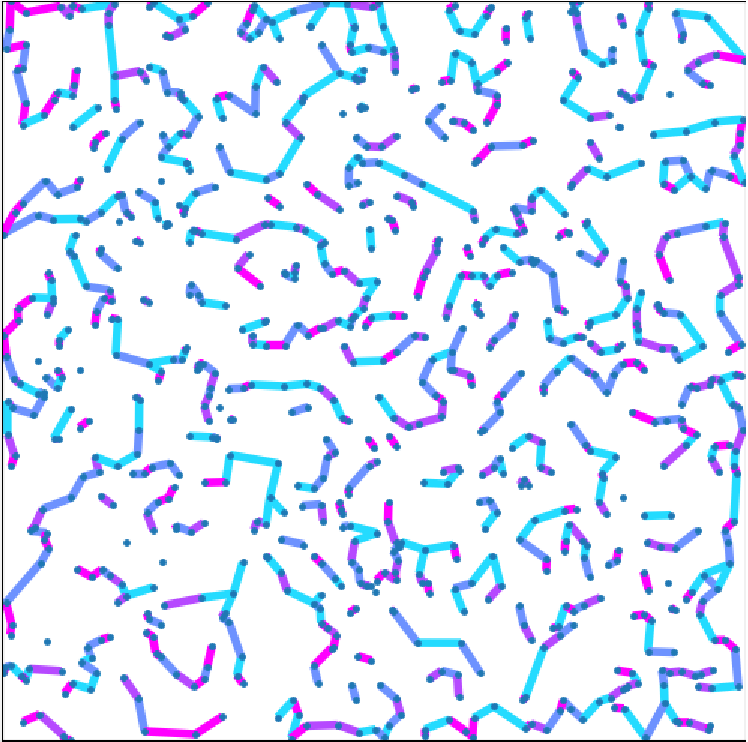}}\\
    \centering
    \subfigure{\includegraphics[width=0.7\textwidth]{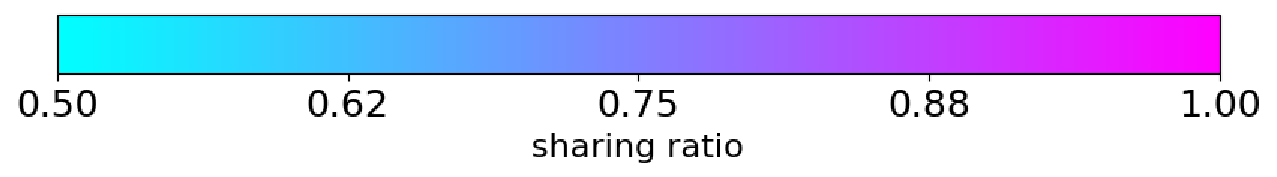}}
    \caption{Superimpositions of the three biggest clusters obtained by SCS in the first experimental instances. (a)--(c) show the results of TSP-100 and (d)--(f) depict the results of TSP-900. Numbers in subcaptions indicate cluster sizes. The color bar at the bottom shows the sharing ratio of edges in each cluster.}
    \label{fig:TSP-clusters}
\end{figure}
(a) covers about 55\% of all solutions, and other clusters are smaller than 1/10 of this size.
Comparing (a)--(c) reveals clusters with similar central small structures but slightly different outer big ones, indicating overall similarity with partial structural variations.
This trend becomes more pronounced in TSP-900.
While the largest cluster is similar to TSP-100, the second-largest cluster contains only less than 1/30 of the largest one, indicating a lack of clusters with distinctive structures.
The absence of large connected components suggests that the sampled solutions display uniform variation in their partial structures, consistent with our discussion in Section \ref{sec:results}.

Based on these findings, we discuss the practical utility of SCS or clustering strategies for TSP.
In TSP-100, we can observe that solutions exhibit high degrees of freedom in their internal structures except for small connected components, whereas outer ones are mainly similar with only minor differences.
This overview of local structures contributes to resolving Issue \ref{iss:suspect}.
Furthermore, identifying partial components provides detailed insights, such as recommending to adopt an entire component when the user prefers specific edges, which helps address Issue \ref{iss:blackbox}.
On the other hand, while TSP-900 similarly reveals partial components, their sizes are relatively small compared to the problem size, making it difficult to grasp overall trends from these figures.
In summary, SCS extracted local insights about variable co-occurrences and provided an overview of local optima for smaller problems.
However, for larger problems, the figures were not helpful for users due to the symmetrical nature of sampled solutions with no discernible structural bias.

Finally, SPP results are shown in Figure \ref{fig:SPP-clusters}.
\begin{figure}[H]
    \centering
    \footnotesize
    \subfigure[size--743]{\includegraphics[width=0.27\textwidth]{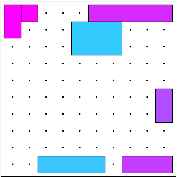}}
    \subfigure[size--44]{\includegraphics[width=0.27\textwidth]{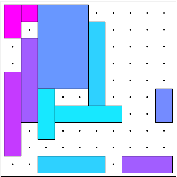}}
    \subfigure[size--34]{\includegraphics[width=0.27\textwidth]{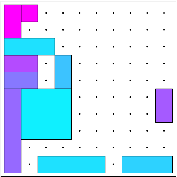}}\\
    \centering
    \subfigure[size--441]{\includegraphics[width=0.27\textwidth]{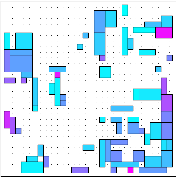}}
    \subfigure[size--23]{\includegraphics[width=0.27\textwidth]{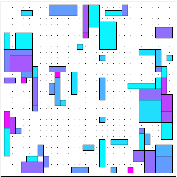}}
    \subfigure[size--16]{\includegraphics[width=0.27\textwidth]{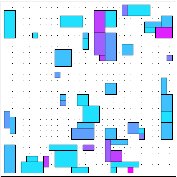}}\\
    \centering
    \subfigure{\includegraphics[width=0.7\textwidth]{cool_colormap_SPP.eps}}
    \caption{Superimpositions of the three biggest clusters obtained by SCS in the first experimental instances. (a)--(c) show the results of SPP-100 and (d)--(f) depict the results of SPP-900. Black points represent elements in the SPP, with rectangles indicating packing sets. Numbers in subcaptions indicate cluster sizes. The color bar at the bottom shows the sharing ratio of edges in each cluster.}
    \label{fig:SPP-clusters}
\end{figure}
(a) shows common structures of the solutions.
The others depict different ones: (b) has a $1\times5$ rectangle with a high sharing ratio positioned on the left, and (c) contains the same size rectangle but positioned one level lower.
Furthermore, the sharing ratio of these rectangles shows gradient variations.
Since solutions with fewer gaps tend to be local optimal, this pattern can represent the co-occurrence of elements, which addresses Issue \ref{iss:blackbox}.
In SPP-900, we can observe substructures or connected elements but the overall structure is unclear, similar to TSP-900.
Moreover, we observe gradient coloring patterns similar to those identified in SPP-100.
These findings may address Issues \ref{iss:not_opt}--\ref{iss:suspect} similar to SCS in TSP.
In conclusion, clustering solutions have unique advantages over existing methods for solving combinatorial optimization problems, which may address Issues \ref{iss:not_opt}--\ref{iss:suspect}.
While this paper deals with problems where cluster superimpositions can be visualized, the discussion can be extended to the general case by directly calculating sharing ratios of variable assignments.

\section{Conclusion}
\label{sec:con}
In this paper, we propose Sample-Cluster-Select (SCS) framework.
It addresses several challenges in obtaining diverse solution sets for combinatorial optimization problems.
This framework provides diverse solutions and valuable insights about the solution space through sampling and clustering solutions.
We conducted the numerical experiments with shortest path problems (SP), traveling salesman problems (TSP), and set packing problems (SPP).
These demonstrate both the capabilities and limitations of SCS.
SCS consistently outperformed existing methods for SP and large-scale SPP across various diversity measures.
While SCS did not exceed the performance of EA approaches for small-scale SPP instances and TSP problems, it consistently improved upon simple sampling methods.
A key advantage of SCS is its clustering step.
It can reveal structural patterns and relationships within solution spaces.
The visualization of clustered solutions provides several significant benefits, such as visualization of choices around representative solutions, insights into variable co-occurrences and local structures, and a view of the solution space, which increases user confidence in the presented options.
These features make SCS valuable in practical applications, even when the performance may be inferior to other heuristics.

One limitation of this work is that it is ineffective in problems with similar solution structures, such as large-scale TSP.
Developing specialized sampling techniques for such cases remains an open challenge for future research.

\bibliographystyle{plain}
\bibliography{cite}
\end{document}